\documentclass[12pt,reqno]{amsart}

\usepackage{mathrsfs,amsfonts,amssymb,amsmath}
\usepackage{graphicx,cite}

\newcommand{\cqfd}{{\nobreak\hfil\penalty50\hskip2em\hbox{}
\nobreak\hfil $\square$\qquad\parfillskip=0pt\finalhyphendemerits=0\par\medskip}}
\newcommand{\R}{\mathbb{R}}

\newcommand{\dt}{\partial_t}
\newcommand{\dx}{\partial_x}

\newcommand{\eps}{ \varepsilon}

\numberwithin{equation}{section}

\newtheorem{theorem}{Theorem}[section]
\newtheorem*{theorem*}{Theorem}

\newtheorem{lemma}{Lemma}[section]

\theoremstyle{definition}

\newtheorem{remark}{Remark}[section]

\author{Wentao Cao}
\address{Institute of Applied Mathematics, AMSS, CAS, Beijing 100190, China.}
\email{cwt@amss.ac.cn}

\author{Feimin Huang}
\address{Institute of Applied Mathematics, AMSS, CAS, Beijing 100190, China.}
\email{fhuang@amt.ac.cn}

\author{Dehua Wang}
\address{Department of Mathematics, University of Pittsburgh,
                           Pittsburgh, PA 15260, USA.}
\email{dwang@math.pitt.edu}

\title[Isometric Immersion of Complete Surfaces]
{Isometric Immersion of Complete Surfaces with Slowly Decaying Negative Gauss Curvature}

\date{\today}



\begin{document}

\begin{abstract}
The isometric immersion of two-dimensional Riemannian manifolds or surfaces in the three-dimensional Euclidean space  is a fundamental problem in differential geometry. When the   Gauss curvature is negative, the isometric immersion problem is considered in this paper through the Gauss-Codazzi system for the second fundamental forms.
It is shown that if the Gauss curvature satisfies an integrability condition,   the surface has a global smooth isometric immersion in the three-dimensional Euclidean space even if the Gauss curvature decays very slowly at infinity.
The new idea of the proof is based on the novel observations on the decay properties of the Riemann invariants of the Gauss-Codazzi system. The weighted Riemann invariants are introduced and a comparison principle is applied with properly chosen control functions.
\end{abstract}

\maketitle

%
\section{Introduction}
The isometric  embedding or immersion   of Riemannian manifolds is a fundamental problem in differential geometry and has been studied extensively   since the 19th century.  Roughly speaking, given a Riemannian manifold ($M^n,\mathfrak{g}$), the isometric  embedding or immersion   problem is to seek a mapping ${\bf r}$ into $\mathbb{R}^m$ such that
\begin{equation}\label{mapping}
d{\bf r}\cdot d{\bf r}=\mathfrak{g}.
\end{equation}
Since the number of equations in \eqref{mapping} is $\frac{n(n+1)}{2}$, it is believed that the critical
dimension for isometric  embedding or immersion   is $m=s_n=\frac{n(n+1)}{2}$ (called   the Janet dimension). There have been a number  of remarkable achievements in the case $m>s_n$;  see  Nash \cite{nash1954, nash1956}, Gromov \cite{gromov1,gromov2} and the references therein. It is a classical and challenging problem to study the critical dimensional case, that is, $m=s_n$.
In this direction, for the analytic metric $\mathfrak{g}$, the local embedding problem was solved by Janet \cite{Janet} in 1926 and Cartan \cite{Cartan} in 1927. For the smooth metric $\mathfrak{g}$, the local smooth isometric immersion was studied in Bryant-Griffiths-Yang  \cite{bgy},  Goodman-Yang \cite{GY}, Nakamura-Maeda \cite{NM,NM2}, Poole \cite{Poole}, and Chen-Clelland-Slemrod-Wang-Yang \cite{CCSWY1} when $n=3$ and in \cite{bgy,GY} when $n=4$.
For more discussions see \cite{CCSWY1} for $n\ge 3$ and \cite{HH} for $n=2$.

In this paper, we focus on the two-dimensional problem, i.e. $n=2, s_n=3$. In this  extremely interesting case, the problem can be reduced to solve the  Darboux equation/Gauss-Codazzi system (c.f.  \cite{BS, HH, PS}).  It is well known that the  Darboux equation/Gauss-Codazzi system is of mixed type of partial differential equations depending on the sign of the Gauss curvature $K$; that is,    the system is elliptic   if $K>0$, hyperbolic if $K<0$, and is of mixed type if $K$ changes sign. There have been considerable progresses on the local smooth embedding \cite{Han,hhl,hw,Lin0,Lin} and the global smooth
embedding \cite{gl,he,hz,n,w} (for the case $K\ge 0$).  Dong \cite{Dong} also studied a semi-global isometric immersion.  We refer the reader to the book of Han-Hong \cite{HH} for an excellent review and discussions of the smooth isometric embedding or immersion of two-dimensional manifolds (surfaces) into the three-dimensional Euclidean space.
We remark that there are some recent works  on the $C^{1, 1}$ isometric immersions via compensated compactness in \cite{CHW, CHW1, CSW, Christoforou, cs}.

The present paper is concerned with the smooth isometric embedding/immersion of surfaces with negative Gauss curvatures. When the Gauss curvature is negative, i.e., $K<0$, the global embedding problem was first investigated by Hilbert \cite{hi}, where he gave a negative answer for the hyperbolic plane whose Gauss curvature is a negative constant.  Efimov  showed  in \cite{Efimov1} that there is no $C^2$ isometric immersion in $\R^3$ if the Gauss curvature is bounded above by a negative constant, and then showed further in   \cite{Efimov2}   that if the Gauss curvature satisfies $$\left|\nabla\frac{1}{k}\right|\leq q_1, \text{for some constant } q_1>0,$$
where $K=-k^2$,  there is no $C^3$ isometric immersion in $\R^3$.
 It is very  challenging to prove the existence of global smooth isometric immersion of a two-dimensional Riemannian manifold with negative Gauss curvature in $\R^3$.
Yau  \cite{Yau} proposed the following problem: {\it ``Find a sufficient  condition for a complete negative curved surface to be isometrically embedded in $\R^3$"}.
Hong  first gave a positive answer in his groundbreaking work \cite{H} as follows.

\begin{theorem*}[{\bf Hong's Theorem in \cite{H}, 1993}]
Let $(\mathcal{M}, \mathfrak{g})$ be a smooth complete simply connected 2-dimensional surface with negative Gauss curvature
$K=-k^2(\theta, \rho).$ Assume that $k$ satisfies
\begin{equation}\label{c1}
\partial_\rho\ln(k\rho^{1+\delta})\leq0 \text{ as } \rho\geq R,
\end{equation}
for some positive constants $\delta$ and $R,$ and
$$\partial_\theta^i\ln k~(i=1,2), ~~\rho\partial_\theta\partial_\rho\ln k \text{ bounded.}$$
Then $(\mathcal{M}, \mathfrak{g})$ has a smooth isometric immersion in $\mathbb{R}^3.$
\end{theorem*}

\begin{remark} In the above theorem, $(\theta, \rho)$ denotes the geodesic polar coordinates. From the  condition (\ref{c1}), the decay rate of the Gauss curvature $k$ at  infinity is
\begin{equation}\label{d}
k\approx\frac{1}{\rho^{1+\delta}}, \text{ for some } \delta>0.
\end{equation}
 On the other hand, Efimov \cite{Efimov2} showed that there is no $C^3$ isometric immersion if  $k=\frac1\rho.$ Thus Hong's decay rate (\ref{c1})  is almost optimal!
\end{remark}

Although Hong's decay rate is almost optimal, there is still a gap between $\frac1\rho$ and $\frac{1}{\rho^{1+\delta}}$. For instance, the case that \begin{equation}\label{c2}
k=\frac{1}{\rho(\ln{\rho})^2} \; \mbox{ or equivalently } \; K=-\frac{1}{\rho^2(\ln{\rho})^4}\; \mbox{for $\rho$  large,}
\end{equation} is excluded in Hong \cite{H}.
How to relax the decay rate  (\ref{c1}) has become  a longstanding open problem since 1993.
In fact, Hong in \cite{H} raised himself the following question: {\it ``One wonders if the restriction on the rate of the decay of the curvature as fast as the geodesic distance from the base curve of $-(2+\delta)$ order at infinity, can be relaxed. It is still open".}  The purpose of this paper is to relax Hong's decay rate (\ref{c1}) for global smooth isometric immersion.
Precisely, our main result can be stated as follows.

\begin{theorem}[{\bf Main Theorem}] \label{thm}
Let ($\mathcal{M},\mathfrak{g}$) be a smooth complete simply connected 2-dimensional surface with the metric $\mathfrak{g}=G^2(\theta, \rho)d\theta^2+d\rho^2$ in the geodesic polar coordinate $(\theta,\rho)$,   the Gauss curvature of which $K=-k^2(\theta, \rho)$ satisfies the following conditions:
\begin{itemize}
    \item[(A1)] $\sup_\theta\int_{ \rho\geq R} k(\theta, \rho)d\rho<\infty,$ $k=o(\frac{1}{\rho}),$ and
        $\partial_\rho\ln(k\rho^{1+\delta})\geq0,$ when $\rho\geq R,$
        for some  positive constants $R$ and $\delta\in(0, 1/2);$
    \item[(A2)] $\partial^i_\theta\ln k, $ $i=1, 2, $  $\rho\partial_\theta\partial_\rho\ln k$ are bounded,
\end{itemize}
Then $(\mathcal{M}, \mathfrak{g})$ has a smooth isometric immersion in $\mathbb{R}^3.$
\end{theorem}

\begin{remark}
The notation $o(\frac{1}{\rho})$ in the above theorem means $\rho\  o(\frac{1}{\rho})\to 0$ as $\rho\to\infty$.
\end{remark}
\begin{remark}
From (A1), 	the case (\ref{c2}) is included in Theorem \ref{thm}. Thus we give a positive answer
to the question raised by Hong \cite{H}.
\end{remark}
\begin{remark}
Since the condition $\partial_\rho\ln(k\rho^{1+\delta})<0$ was treated in \cite{H},  we
consider the other side  $\partial_\rho\ln(k\rho^{1+\delta})\geq0$ where the Gauss curvature may decay slowly.
\end{remark}
\begin{remark}
	As suggested by Yau \cite{Yau}, the decay rate of curvature shoud be replaced by an integrability condition. The integrability condition $\sup_\theta\int_{ \rho\geq R} k(\theta, \rho)d\rho<\infty$ of (A1) is derived from an observation of the  special solutions to  the Gauss-Codazzi system, see Section \ref{ode} below. This integrability condition might be optimal in terms of the decay rate of Gauss curvature at infinity.
\end{remark}

As remarked above, the main contribution of this paper is that the decay rate of Hong \cite{H} is relaxed so that a curvature like \eqref{c2} between the gap of $ \frac1\rho$ and $\frac1{\rho^{1+\delta}}$ is included in our existence theorem of global smooth isometric immersion of surfaces in $\mathbb{R}^3$.
To prove the result in Theorem \ref{thm}, new ideas are developed based on our novel observations  and insights into the behavior of the Riemann invariants of the Gauss-Codazzi system.
We now explain the main strategy. First, we consider a special case that the metric only depends on one variable $t$, thus the Gauss-Codazzi system is reduced to an ODE system for
the Riemann invariants $(w,z)$. Then we obtain an explicit formula of solution to the ODE system. From the explicit formula, we have two new important observations:
(a)   when the curvature $k$ decays slowly than $t^{-2}$,
both $w$ and $z$ decays at the same order as $k$, i.e., $w\approx k,z\approx k$ (see Section \ref{ode} below);
(b)    $w+z\approx t^{-2}$ decays faster than $w-z\approx k$, which means that some cancellations happen in $w+z$. We expect that the above two properties still hold for general cases. The two observations are the crucial motivations for our approach of proving Theorem \ref{thm}.  First novel ingredient of our proof is that we introduce two weighted
Riemann invariants $r=\frac{w}{k}, s=\frac{z}{k}$ and derive an equivalent system
(\ref{rs}) for $(r,s)$. Based on the comparison principle for linear hyperbolic system, we can show that the first property   (a) holds under some additional  {\it a priori}   assumption of derivatives  $\tilde{r}=(r-s)r_x, \tilde{s}=(r-s)s_x$.
The   {\it a priori}    assumption on the derivatives  can then be verified by the fact that $r+s$ decays faster than $r-s$ at infinity, which is another novel part of our proof.
Using the approach above we can finally show that there is a global smooth isometric immersion of surfaces with slowly decaying Gauss curvature, which relaxes the decay rate of curvatures in Hong \cite{H}.

The rest of the paper is organized as follows. In Section \ref{ode} we provide some basic formulas related to the Gauss-Codazzi system for surfaces with negative Gauss curvature. Then we give an explicit formula of special solutions to the Gauss-Codazzi system. The explicit formula of the special solutions provides us important observations that motivate our new ideas.  Section \ref{geodesic} is devoted to
a key   theorem for the existence of global smooth isometric immersion of surfaces in the geodesic coordinates. This section contains our main new ideas and show why we can relax the decay rate.
The   main Theorem \ref{thm} in geodesic polar coordinates is proved in  Section \ref{geodesicpolar} following the original idea of Hong \cite{H} with modifications through variable transformations.
\bigskip

%
\section{Special Solution to the Gauss-Codazzi System}\label{ode}
%
The Gauss-Codazzi system for the isometric immersion or embedding of
 surfaces into $\mathbb{R}^3$ is (cf. \cite{HH, H}):
\begin{equation}\label{GC}
\begin{split}
& \partial_{x_2}L-\partial_{x_1}M=\Gamma^1_{12}L+(\Gamma^2_{12}-\Gamma^1_{11})M-\Gamma^2_{11}N,\\
 &\partial_{x_2}M-\partial_{x_1}N=\Gamma^1_{22}L+(\Gamma^2_{22}-\Gamma^1_{21})M-\Gamma^2_{21}N,\\
 &LN-M^2=K|\mathfrak{g}|,
 \end{split}
\end{equation}
where $L, M$ and $N$  denote the coefficients  of the  second fundamental form
 \begin{equation}\label{secondform}
 I\!\!I= Ldx_1^2+2Mdx_1dx_2+Ndx_2^2,
 \end{equation}
 with the given metric $\mathfrak{g}=g_{ij}dx_idx_j, i,j=1,2,$ of the surface;
$\Gamma^i_{jk}, i,j,k=1,2,$ are the Christoffel symbols,  $|\mathfrak{g}|$ is the determinant
of the metric matrix $(g_{ij}),$  and $K$ is the Gauss curvature.

In this paper, we consider the isometric immersion of a  complete simply connected two-dimensional  manifold (or a surface) with the negative Gauss curvature
$$K=-k^2$$ for some smooth function $k>0$.
Then the Gauss-Codazzi system \eqref{GC} is  hyperbolic,  the two eigenvalues are
$$\lambda_+=\frac{-M+k\sqrt{|\mathfrak{g}|}}{L}, \quad \lambda_-=\frac{-M-k\sqrt{|\mathfrak{g}|}}{L},$$
and the corresponding Riemann invariants (cf. \cite{Smoller})   are
\begin{equation*}
w=\frac{-M+k\sqrt{|\mathfrak{g}|}}{L},\quad  z=\frac{-M-k\sqrt{|\mathfrak{g}|}}{L}.
\end{equation*}
 For simplicity of notations, let $x=x_1,$ $t=x_2.$
A direct calculation leads to the following system for the Riemann invariants $w$ and $z$ (cf.  \cite{HH}):
\begin{equation}\label{wz}
\begin{split}
w_t+zw_x=&\frac{w-z}{2}(\partial_t+w\partial_x)\ln k-\Gamma^1_{22}+(\Gamma^2_{22}-\Gamma^1_{12})w\\
&-\Gamma^1_{12}z+\Gamma^2_{12}w^2+(\Gamma^2_{12}-\Gamma^1_{11})wz+\Gamma^2_{11}w^2z,\\
z_t+wz_x=&\frac{z-w}{2}(\partial_t+z\partial_x)\ln k-\Gamma^1_{22}+(\Gamma^2_{22}-\Gamma^1_{12})z\\
&-\Gamma^1_{12}w+\Gamma^2_{12}z^2+(\Gamma^2_{12}-\Gamma^1_{11})wz+\Gamma^2_{11}wz^2.
\end{split}
\end{equation}
The system (\ref{wz}) is equivalent to (\ref{GC}) when $w>z.$
Note that  (\ref{wz}) is  linearly degenerate. 
We shall show that the system \eqref{wz} admits a global smooth solution under certain conditions.

%
In particular, in this section we explore some special solution to the Gauss-Codazzi system, which provides us interesting observations and motivations on the conditions for the global existence of smooth solutions.
%
Consider a special case that the metric is of the form $\mathfrak{g}=B^2(t)dx^2+dt^2,$ then the  corresponding Christoffel symbols  are
\begin{equation}\label{christoffel1}
\begin{split}
&\Gamma^1_{12}=\frac{B'}{B}, \quad \Gamma^2_{11}=-BB',\\
&\Gamma^1_{11}=\Gamma^1_{22}=\Gamma^2_{12}=\Gamma^2_{22}=0,
\end{split}
\end{equation}
where $'$ denotes $\frac{d}{dt}.$ When the initial data  of the system \eqref{wz} is given by $w(0)=w_0, z(0)=z_0$ with $w_0$ and $z_0$ both constant,  we  consider the special solution $(w, z)(t)$ to (\ref{wz}).
Plugging   (\ref{christoffel1}) into (\ref{wz}), we obtain the following ODE system:
 \begin{equation}\label{Riemann ODE}
\begin{split}
&\frac{dw}{dt}=-\left(\frac{B'}{B}-\frac{k'}{2k}\right)w-\left(\frac{B'}{B}+\frac{k'}{2k}\right)z-BB'w^2z,\\
&\frac{dz}{dt}=-\left(\frac{B'}{B}-\frac{k'}{2k}\right)z-\left(\frac{B'}{B}+\frac{k'}{2k}\right)w-BB'wz^2.
\end{split}
\end{equation}
We now solve the ODE system  (\ref{Riemann ODE}) that has  cubic  nonlinear source terms.
Note that 
adding and subtracting the two equations of \eqref{Riemann ODE} yield
\begin{equation}\label{ode1a}
\begin{split}
&\frac{d(w+z)}{dt}=-\frac{2B'}{B}(w+z)-BB'wz(w+z),\\
&\frac{d(w-z)}{dt}=\frac{k'}{k}(w-z)-BB'wz(w-z).
\end{split}
\end{equation}
Then we get
\begin{equation*}
\begin{split}
&w=\left(\frac{w_0+z_0}{2B^2}+\frac{k(w_0-z_0)}{2k(0)}\right)\exp\left[\int_0^t-BB'wzds\right],\\
&z=\left(\frac{w_0+z_0}{2B^2}-\frac{k(w_0-z_0)}{2k(0)}\right)\exp\left[\int_0^t-BB'wzds\right].
\end{split}
\end{equation*}
It holds that
\begin{equation*}
wz=\left(\frac{(w_0+z_0)^2}{4B^4}-\frac{k^2(w_0-z_0)^2}{4k^2(0)}\right)\exp\left[\int_0^t-2BB'wzds\right].
\end{equation*}
Let $X:=wz$. We have
\begin{equation*}
\frac{d\ln X}{dt}=\frac{d\ln a(t)}{dt}-2BB'X, \quad a(t)=\frac{(w_0+z_0)^2}{4B^4}-\frac{k^2(w_0-z_0)^2}{4k^2(0)},
\end{equation*}
then
\begin{equation*}
X(t)=\frac{\frac{(w_0+z_0)^2}{4B^4}-\frac{k^2(w_0-z_0)^2}{4k^2(0)}}
{1+\frac{1}{4}(w_0+z_0)^2(1-\frac{1}{B^2})-\frac{1}{4}(w_0-z_0)^2B^{'2}},
\end{equation*}
where we have used the Gauss equation $B''=k^2B.$
Thus we obtain a special solution to the Gauss-Codazzi system (\ref{wz}) of the following form:
\begin{equation}\label{w}
w=\frac{k(0)(w_0+z_0)+B^2k(w_0-z_0)}
{\sqrt{4B^4k(0)^2+k^2(0)(w_0+z_0)^2(B^4-B^2)-k^2(0)(w_0-z_0)^2B^4B^{'2}}},
\end{equation}
\begin{equation}\label{z}
z=\frac{k(0)(w_0+z_0)-B^2k(w_0-z_0)}
{\sqrt{4B^4k(0)^2+k^2(0)(w_0+z_0)^2(B^4-B^2)-k^2(0)(w_0-z_0)^2B^4B^{'2}}}.
\end{equation}

When $k$ decays at a slower rate than $t^{-2}$  for large $t$, we have the following two important observations:
\begin{itemize}
\item[(1)] both $w$ and $z$ decays at the same order as $k$, i.e., $w\approx k,z\approx k$;
\item[(2)] $w+z\approx t^{-2}$ decays faster than $w-z\approx k$.
\end{itemize}
On the other hand, from   (\ref{w}) and (\ref{z}),
the solution may blow up in a finite time if $B'$ is not bounded. Note that  $B'$ has the following estimate as in \cite{H}:
$$\int_0^tk^2dt\leq B'(t)\leq\int_0^tk^2ds\exp\left[\int_0^tsk^2ds\right].$$
Therefore the ODE system (\ref{Riemann ODE}) has a global solution only if
\begin{equation}\label{cond1}
\int_0^\infty k^2(t)tdt<+\infty.
\end{equation}
It seems that Hong's decay rate (\ref{d}) could be relaxed by (\ref{cond1}).

However, the metric generally  depends not only on $t$ but also on $x.$
In general,  the Christoffel symbols  are
\begin{equation}\label{christoffel2}
\begin{split}
&\Gamma^1_{11}=\frac{B_x}{B}, \quad \Gamma^2_{11}=-BB_t,\quad \Gamma^1_{22}=0,\\
&\Gamma^1_{12}=\frac{B_t}{B}, \quad \Gamma^2_{12}=\Gamma^2_{22}=0,
\end{split}
\end{equation}
then  the Gauss-Codazzi system (\ref{wz}) contains some quadratic terms $w^2, wz, z^2$. Note that
$\partial_x \ln B=\frac{B_x}{B}$ and $\partial_x \ln k=\frac{k_x}{k}$ are uniformly bounded (see Lemma \ref{lemmaB} below), all the coefficients of quadratic terms in the Gauss-Codazzi system \eqref{wz} are uniformly bounded.
We consider the following toy model of  (\ref{wz}), i.e. the ODE system (\ref{Riemann ODE}) with additional quadratic terms:
\begin{equation}\label{ode1}
	\begin{split}
		\frac{dw}{dt}=&-\left(\frac{B_t}{B}-\frac{k_t}{2k}\right)w-\left(\frac{B_t}{B}
		+\frac{k_t}{2k}\right)z-BB_tw^2z+w^2,\\
		\frac{dz}{dt}=&-\left(\frac{B_t}{B}-\frac{k_t}{2k}\right)z-\left(\frac{B_t}{B}
		+\frac{k_t}{2k}\right)w-BB_twz^2-z^2,
	\end{split}
\end{equation}
which can be reduced to a scalar equation
\begin{equation}\label{ode2}
\frac{dw}{dt}=\frac{k_t}{k}w-BB_tw^2z+w^2,
\end{equation}
for the special case $w=-z$.
It is not difficult to show that the equation \eqref{ode2} has a global solution only if $\int^\infty k(t)dt<\infty.$
So we instead use
\begin{equation}\label{cond2}
\sup_x\int_0^\infty k(t)dt<+\infty
\end{equation}
to relax Hong's decay rate (\ref{d}).
Such decay rate  in (\ref{cond2})  might be optimal for admitting a global immersion in $\mathbb{R}^3$, based on the above observation of special solution to the Gauss-Codazzi system (\ref{GC}).

\bigskip

%
\section{Isometric Immersion in Geodesic Coordinates}\label{geodesic}
%
In this section, we prove the isometric immersion in geodesic coordinates under the similar conditions of Theorem \ref{thm} based on the two important observations in the previous Section \ref{ode}.   This section contains our main new ideas and contributions and show   why we can relax the decay rate in  Hong \cite{H} of Gauss curvature of surfaces for the global existence of smooth isometric immersion.

From Theorem B in \cite{H}, the geodesic coordinate system covers the whole complete simply connected two-dimensional surface with a negative Gauss curvature.
The metric of the two-dimensional surface  in  the geodesic coordinates is  of the form
$$\mathfrak{g}=B^2(x, t)dx^2+dt^2,$$  where $B(x, t)>0$ satisfies
\begin{eqnarray}\label{gauss}
\left\{
\begin{array}{lllll}
B_{tt}=k^2B,\\
B(x,0)=1, \quad B_t(x,0)=0.
\end{array}
\right.
\end{eqnarray}

We have the following theorem on the global isometric immersion in the geodesic coordinates.
\begin{theorem} \label{thm1}
Let $(\mathcal{M},\mathfrak{g})$ be a complete simply connected smooth two-dimensional surface with the metric $\mathfrak{g}=B^2(x,t)dx^2+dt^2$
and the Gauss curvature $K=-k^2(x,t)$ satisfying
\begin{itemize}
    \item[(H1)]  $\sup_x\int_{ |t|\geq T} k(x,t)dt<\infty,$ $k=o\left(\frac{1}{|t|}\right),$
                               $t\dt\ln k(|t|^{1+\delta})\geq0$  when $|t|\geq T$
                               for some  positive constant $T$ and $\delta\in(0, 1/2);$
    \item[(H2)]  $\partial^i_x\ln k, $ $i=1, 2, $  $t\dx\dt\ln k$ are bounded;
    \item[(H3)]  $\inf_x\int_0^\infty k^2(x,t)dt$ and $\inf_x\int_{-\infty}^0k^2(x,t)dt$ are positive.
\end{itemize}
Then $(\mathcal{M},\mathfrak{g})$ admits a smooth isometric immersion in $\mathbb{R}^3.$
\end{theorem}

\begin{remark}
The condition (H1) is motivated by the observations in Section \ref{ode}. Comparing with $(H_1)$ of Theorem C in \cite{H}:
\begin{itemize}
    \item[$(H_1)$]  $k>0$ and $t\dt\ln k(|t|^{1+\delta})\leq0,$  when $|t|\geq T,$
                  for some large enough positive constant $T$ and $\delta\in(0, 1),$
\end{itemize}
the condition (H1) in Theorem \ref{thm1} allows slower decay rate for the Gauss curvature.
 For instance, let
\begin{equation}\label{ex}
k=\frac{1}{(t+2)(\ln(t+2))^2}, \quad t>>1,
\end{equation} it is straightforward to check that (\ref{ex}) is included in (H1) of Theorem \ref{thm1}, while excluded in
$(H_1)$ of Theorem C in \cite{H}. Thus Theorem \ref{thm1} gives a positive answer to the question raised in \cite{H}:  ``{\it One wonders if the restriction on the rate of the decay of the curvature as fast as the geodesic distance from the base curve of $-(2+\delta)$ order at infinity, can be relaxed.} "
\end{remark}

We first derive some properties of $B(x, t)$ under the conditions of Theorem \ref{thm1}.
\begin{lemma}\label{lemmaB}
Under the conditions of Theorem \ref{thm1}, one has the following estimate:
\begin{equation*}
\begin{split}
&\frac{B_t}{B}=\frac{1}{t}+o\left(\frac{1}{|t|}\right) \text{ uniformly in $x$ for $|t|$ sufficiently large, and}\\
& \dx^i\ln B, i=1, 2,~B\dx\dt\ln B\text{ are bounded for $|t|$ sufficiently large.}
\end{split}
\end{equation*}
\end{lemma}
\begin{proof}
We only consider the case $t\geq0$ since $t\le 0$ can be treated similarly.  From Lemma 1.1 in \cite{H}, we have for fixed $x,$
\begin{equation*}
\int_0^tk^2(x,t)dt\leq B_t\leq\int_0^tk^2(x,s)ds\exp\left[\int_0^tsk^2(x,s)ds\right].
\end{equation*}
On the other hand, by (H1) and (H3), one has
\begin{equation*}
\begin{split}
&\sup_x\int_0^tk^2(x,s)sds<C\sup_x\int_0^tk(x,s)ds<\infty,\\
&\inf_x\int_0^tk^2(x,t)dt\leq B_t\leq
\sup_x\int_0^tk^2(x,s)ds\exp\left[\sup_x\int_0^tsk^2(x,s)ds\right],
\end{split}
\end{equation*}
then there exist positive constant $C_1$ and $C_2$ such that
$$C_1|t|\leq B(x,t)\leq C_2|t|$$
 for large $|t|.$ In addition, from (\ref{gauss}) we have
\begin{equation}\label{ga1}
\partial_t\partial_t\ln B=k^2-(\partial_t\ln B)^2,
\end{equation}
which implies that
$$|\partial_t\ln B|\leq\int_0^tk^2(x,s)ds\leq C.$$
Differentiating (\ref{ga1}) with respect to $x$ yields the following equation for $\partial_x\partial_t\ln B(x,t),$
$$\partial_t(\partial_x\partial_t\ln B)=2kk_x-2\partial_x\partial_t\ln B\partial_t\ln B,$$
then
$$\partial_x\partial_t\ln B=\frac{1}{B^2}\int_0^t2kk_xB^2ds.$$
Noting that $B_x( x, 0)=0$, we have
\begin{equation*}
\begin{split}
|\partial_x\ln B|&=\left|\int_0^t\frac{1}{B^2}\int_0^s2kk_xB^2d\tau ds\right|
\leq C\int_0^t\frac{1}{B^2(x. s)}\int_0^sk^2(x, \tau)B^2(x, \tau)d\tau ds\\
&= C\int_0^tk^2(x, \tau)B^2(x, \tau)\int_\tau^t\frac{1}{B^2(x. s)}dsd\tau
\leq C\int_0^tk^2(x, \tau)\tau d\tau\leq C,
\end{split}
\end{equation*}
where we have used the condition (H2),  that is,  $|k_x|\leq Ck$. Hereafter $C$  denotes a generic
positive constant. 
Similarly, we have
\begin{equation*}
\begin{split}
|\partial_x\partial_x\ln B|&=\left|\partial_x\left(\int_0^t\frac{1}{B^2}\int_0^s2kk_xB^2d\tau ds\right)\right|\leq\left|\int_0^t\frac{2B_x}{B^3}\int_0^skk_xB^2d\tau ds\right|\\
&\quad +\left|\int_0^t\frac{1}{B^2}\int_0^s\left(2k_xk_xB^2+2kk_{xx}B^2
+4kk_xBB_x\right)d\tau ds\right|\\
&\leq C\int_0^t\frac{1}{B^2}\int_0^sk^2B^2d\tau ds\leq C,
\end{split}
\end{equation*}
and
$$|B\partial_x\partial_t\ln B|\leq \frac{1}{B}\int_0^t2k|k_x|B^2ds\leq C\int_0^tsk^2(x,s)ds\leq C.$$
Thus,  $\partial_x^i\ln B, i=1, 2$ and $B\partial_x\partial_t\ln B$
are bounded for large $|t|.$\par
Next we shall study the asymptotic behavior of $\partial_t\ln B$ for large $t$. Since $B_{tt}>0$ and $B_t$ is bounded, $B_t(x, \infty)$ exists. Moreover,
\begin{equation*}
\begin{split}
&B_t(x,t)=B_t(x,\infty )-\int_t^\infty Bk^2(x,s)ds,\\
&B(x,t)=1+B_t(x,\infty)t-t\int_t^\infty k^2Bds-\int_0^tk^2Bsds.
\end{split}
\end{equation*}
Then
\begin{equation*}
\begin{split}
\frac{B_t}{B}&=\frac{B_t(x,\infty )-\int_t^\infty Bk^2(x,s)ds}
{1+B_t(x,\infty)t-t\int_t^\infty k^2Bds-\int_0^tBk^2sds}\\
&=\frac{1}{t}+\frac{1}{t}\frac{t^{-1}\int_0^tBk^2sds-t^{-1}}{t^{-1}+B_t(x,\infty)-\int_t^\infty k^2Bds-t^{-1}\int_0^tBk^2sds}.
\end{split}
\end{equation*}
From condition (H1),  $kt^{1+\delta}$ is increasing for large $t$ with $\delta\in(0, \frac12).$ Then it follows from $k=o(\frac{1}{|t|})$ that
\begin{equation}\label{ktdecay}
\lim_{t\rightarrow\infty}t^{-1}\int_0^tk^2s^2ds\leq \lim_{t\rightarrow\infty}k^2t^{1+2\delta}\int_0^ts^{-2\delta}ds\le \lim_{t\rightarrow\infty}\frac{1}{1-2\delta}k^2t^2=0.
\end{equation}
Noting that $$B_t(x, \infty)\geq\inf_x\int_0^\infty k^2(x, s)ds>0,$$ one has
$$\lim_{t\rightarrow\infty}\frac{t^{-1}\int_0^tBk^2sds-t^{-1}}{t^{-1}+B_t(x,\infty)-\int_t^\infty k^2Bds-t^{-1}\int_0^tBk^2sds}=0, $$
that is,  $\frac{B_t}{B}=\frac{1}{t}+o\left(\frac{1}{|t|}\right)$.
This completes the proof of Lemma \ref{lemmaB}.
\end{proof}

Now we turn to the proof of Theorem \ref{thm1}. Plugging  all the  Christoffel symbol formulas \eqref{christoffel2} into (\ref{wz}), we obtain
 \begin{equation}\label{Riemann}
 \begin{split}
&w_t+zw_x=-\left(\frac{B_t}{B}-\frac{k_t}{2k}\right)w-
\left(\frac{B_t}{B}+\frac{k_t}{2k}\right)z-BB_tw^2z
-\left(\frac{B_x}{B}+\frac{k_x}{2k}\right)wz+\frac{k_x}{2k}w^2,\\
&z_t+wz_x=-\left(\frac{B_t}{B}-\frac{k_t}{2k}\right)z-
\left(\frac{B_t}{B}+\frac{k_t}{2k}\right)w-BB_twz^2
-\left(\frac{B_x}{B}+\frac{k_x}{2k}\right)wz+\frac{k_x}{2k}z^2.
\end{split}
\end{equation}
Motivated by the observation   in Section \ref{ode}  that $w, z$ decay at the same order as $k$ for the ODE system (\ref{Riemann ODE}),   we make the following variable transformation to introduce the weighted Riemann invariants:
$$r=\frac{w}{k}\quad \text{ and } \quad  s=\frac{z}{k},$$
then \eqref{Riemann} is rewritten as
\begin{equation}\label{rs}
\begin{split}
r_t+ksr_x&=-\left(\frac{B_t}{B}+\frac{k_t}{2k}\right)(r+s)-\left(\frac{B_x}{B}+\frac{3k_x}{2k}\right)krs+
                        \frac{k_x}{2k}kr^2-BB_tk^2r^2s\\
                        &=:f(r, s, x, t),\\
s_t+krs_x&=-\left(\frac{B_t}{B}+\frac{k_t}{2k}\right)(r+s)-\left(\frac{B_x}{B}+\frac{3k_x}{2k}\right)krs+
                          \frac{k_x}{2k}ks^2-BB_tk^2rs^2\\
                          &=:g(r, s, x, t).
\end{split}
\end{equation}
The system  (\ref{wz}) is equivalent to \eqref{rs} if $r>s$ holds.  We focus on the system \eqref{rs} from now on. Denote
 \begin{equation*}
 \partial_\beta=\partial_t+ks\partial_x, \quad  \partial_\alpha=\partial_t+kr\partial_x,
 \end{equation*}
then we have the following system  along the characteristics:
\begin{equation}\label{chr}
\partial_\beta r=f,\quad \partial_\alpha s=g.
\end{equation}
We now derive a  system for $r-s$ similar to the ODE system \eqref{ode1}.
Substracting the second equation from the first one in the system \eqref{rs} yields
\begin{equation}\label{r-s}
\begin{split}
&\partial_\beta(r-s)=(r-s)Q+k\tilde{s},\\
&\partial_\alpha(r-s)=(r-s)Q+k\tilde{r},
\end{split}
\end{equation}
where
\begin{equation}\label{deri}
\tilde{r}=:(r-s)r_x,~~\tilde{s}=:(r-s)s_x,\end{equation} and
\begin{equation}\label{Q}
Q=\frac{f-g}{r-s}=-BB_tk^2rs+\frac{k_x}{2k}(r+s)k.
\end{equation}
Since the system \eqref{chr} has two characteristics, we have additional derivative terms $\tilde{r}, \tilde{s}$ in \eqref{r-s}. Similarly, we have the following system for $r+s$:
\begin{equation}\label{r+s}
\begin{split}
\partial_\beta(r+s)=&-k\tilde{s}-(r+s)
\left[\left(\frac{2B_t}{B}+\frac{k_t}{k}\right)+BB_tk^2rs\right]\\
&-2\left(\frac{B_x}{B}+\frac{3k_x}{2k}\right)krs+\frac{k_x}{2k}k(r^2+s^2),\\
\partial_\alpha(r+s)=&k\tilde{r}-(r+s)
\left[\left(\frac{2B_t}{B}+\frac{k_t}{k}\right)+BB_tk^2rs\right]\\
&-2\left(\frac{B_x}{B}+\frac{3k_x}{2k}\right)krs+\frac{k_x}{2k}k(r^2+s^2).
\end{split}
\end{equation}
Differentiating (\ref{rs}) with respect to $x$,  we derive the following system for $\tilde{r}$ and $\tilde{s}$:\begin{equation}\label{derivative}
\begin{split}
&\partial_\beta\tilde{r}=\left(Q+f_r-k_xs\right)
\tilde{r}+f_s\tilde{s}+(r-s)\delta_xf,\\
&\partial_\alpha\tilde{s}=g_r\tilde{r}+
\left(Q+g_s-k_xr\right)\tilde{s}+(r-s)\delta_xg,
\end{split}
\end{equation}
where
\begin{equation}\label{fr}
\begin{split}
&f_r(r, s, x, t)=g_s(s, r, x, t)=-\frac{k_t}{2k}-\frac{B_t}{B}-2BB_tk^2rs-
\left(\frac{B_x}{B}+\frac{3k_x}{2k}\right)ks+\frac{k_x}{k}kr,\\
&f_s(r, s, x, t)=g_r(s, r, x, t)=-\frac{k_t}{2k}-\frac{B_t}{B}-BB_tk^2r^2-
\left(\frac{B_x}{B}+\frac{3k_x}{2k}\right)kr,
\end{split}
\end{equation}
and $\delta_xf$ denotes the derivative with respect to $x$ of the coefficients of $r$ and $s$ in $f,$ i.e.
\begin{equation}\label{deltaxf}
\begin{split}
&\delta_xf=-\left(\frac{B_t}{B}+\frac{k_t}{2k}\right)_x(r+s)
-\left[\left(\frac{B_x}{B}+\frac{3k_x}{2k}\right)k\right]_xrs+
\left(\frac{k_x}{2k}k\right)_xr^2-(BB_tk^2)_xr^2s,\\
&\delta_xg=-\left(\frac{B_t}{B}+\frac{k_t}{2k}\right)_x(r+s)
-\left[\left(\frac{B_x}{B}+\frac{3k_x}{2k}\right)k\right]_xrs+
\left(\frac{k_x}{2k}k\right)_xs^2-(BB_tk^2)_xrs^2.
\end{split}
\end{equation}


Now we are ready  to  prove Theorem \ref{thm1}. The strategy is based on the local existence and   {\it a priori}   estimates. The local existence can be obtained  by a lemma of \cite{H} that
holds for general hyperbolic systems. To derive {\it a priori}   estimates, we consider the following hyperbolic system:
\begin{equation}\label{ghs}
\begin{split}
&u_t+\lambda_1u_x=a_{11}u+a_{12}v+R_1,\\
&v_t+\lambda_2v_x=a_{21}u+a_{22}v+R_2,
\end{split}
\end{equation}
with the initial data:
$$u(x,0)=\varphi_1(x), \quad v(x,0)=\varphi_2(x),$$
where   $\lambda_i$,  $a_{ij}$,   $R_i$, ${i,j=1,2},$ are all $C^1$ smooth functions of $(x,t)$, and $\lambda_1\leq\lambda_2$.
For each $\lambda_i,$ from any point $(x,t),$ we can draw a backward characteristic curve
$X=\Gamma_i(\tau;x,t)$ defined by the   following ODE:
\begin{equation}
\frac{dX}{d\tau}=\lambda_i(X,\tau), \, \tau\leq t, \quad \text{with }X(t)=x.
\end{equation}
It is obvious that  $\Gamma_2(\tau;x,t)\leq\Gamma_1(\tau;x,t),$
for $0<\tau<t<T.$
\par
For each point $P=(x^*,t^*),$ denoted by $P(x^*,t^*)$,  $t^*\in(0,T],$ we can draw a backward characteristic triangle $\Delta_P^0,$
\begin{equation*}
\Delta_P^0=\{(x,t): \, \Gamma_2(t;x^*,t^*)\leq x \leq\Gamma_1(t;x^*,t^*), \, 0\leq t\leq t^*\}.
\end{equation*}
Denote
$$H=\max_{(x,t)\in\Delta_P^0}\{|a_{ij}|,|\partial_x\lambda_i|,|\partial_xa_{ij}|, i,j=1,2\},$$
and $I(\tau)=[\Gamma_2(\tau; x^*, t^*),\Gamma_1(\tau; x^*, t^*)].$  Then we have the following lemma from \cite{H}:

\begin{lemma}[Hong \cite{H}] \label{lemmageneral}
	Let $(u,v)$ be a $C^1$ smooth solution to (\ref{ghs}) with $R_i=0, i=1,2$ in $\Delta_P^0.$
	Then for $(x,t)\in\Delta_P^0,$
	\begin{equation}
	|u|, |v|, |u_x|, |v_x|\leq\max_{x\in I(0)}\{|\varphi_i|, |\partial_x\varphi_i|\}\exp(5Ht).
	\end{equation}
\end{lemma}

If we choose sufficiently small  initial data,   the above lemma tells us that the lifespan of the smooth solutions could be large and the solutions still keep small in the region of local existence. Based on this property, we begin to
derive the  {\it a priori}   estimates in the region $t\ge t_0$ with large $t_0$ to be determined.
For each $P(x^*,t^*),$ $t^*>t_0$, we  denote
\begin{equation*}
\Delta_P=\{(x, t): \, \Gamma_2(t; x^*, t^*)\leq x \leq\Gamma_1(t; x^*, t^*), \, t_0\leq t\leq t^*\},
\end{equation*} and  $I(t_0)=[\Gamma_2(t_0) ,\Gamma_1(t_0)],$
where $\Gamma_i(t)=\Gamma_i(t;  x^*, t^*), ~i=1, 2,$ are the characteristic curves corresponding to
$\lambda_i, i=1, 2, $
respectively,  passing through the point $P(x^*, t^*).$

Also we can assume that the system \eqref{rs} has a  $C^1$ smooth solutions with $r-s>0$ satisfying the  {\it a priori}   assumptions:
$$-10a_0\varepsilon\leq s(x, t)<r(x,t)\leq10a_0\varepsilon,\quad
|\partial_xr(x, t)|\leq 10a_0\mu,\quad
|\partial_xs(x, t)|\leq 10a_0\mu,$$
where $\varepsilon,$ $\mu$ are small constants
to be determined later and  $$a_0=2+\sup_x\int_{t_0}^\infty k(x, s)ds<+\infty.$$ From
the fomulas of $\tilde{r},\tilde{s}$, it holds that
\begin{equation}\label{tilde}
|\tilde{r}(x, t)|\leq 200a_0^2\mu\varepsilon,\quad |\tilde{s}(x, t)|\leq 200a_0^2\mu\varepsilon.
\end{equation}
We now prove  the following key lemma:

\begin{lemma}\label{lemmabound}
Assume that the assumptions {\rm (H1)-(H3)} of Theorem \ref{thm1} hold,  and for any $x\in I(t_0),$
\begin{equation*}
\begin{split}
&|r(x, t_0)|\leq\varepsilon, \quad |s(x, t_0)|\leq\varepsilon,\\
&|\partial_x r(x, t_0)|\leq\mu, \quad |\partial_xs(x, t_0)|\leq\mu.
\end{split}
\end{equation*}
Then,  in $\Delta_P$ for $t>t_0$ and  $t_0$   large enough,
\begin{equation*}
\begin{split}
&|r(x, t)|\leq a_0\varepsilon, \quad |s(x, t)| \leq a_0\varepsilon,\\
&|\tilde{r}(x, t)|\leq a_0\mu\varepsilon, \quad |\tilde{s}(x, t)|\leq a_0\mu\varepsilon.
\end{split}
\end{equation*}
\end{lemma}

\begin{proof} We divide the proof into two steps: estimating $r, s$ in Step 1 and $\tilde{r}, \tilde{s}$ in Step 2.

{\bf Step 1.}\quad
We first estimate $r+s$ and $r-s.$
For $r+s,$ we introduce a control function for the system \eqref{r+s}. Let
$$\phi_1=2\varepsilon\frac{k(x, t_0)t_0^2}{k(x,t)t^2}
+\frac{\varepsilon}{k(x, t)t^2}\int_{t_0}^tk^2(x, s)s^2ds,$$
then
\begin{equation*}
\begin{split}
\phi_1&\leq\frac{2\eps k(x, t_0)t_0^{1+\delta}t_0^{1-\delta}}
{k(x, t)t^{1+\delta}t^{1-\delta}}
+\eps \frac{kt^2kt^{2\delta}}{kt^2}\int_{t_0}^ts^{-2\delta}ds\\
&\leq\frac{C\eps}{t^\delta}+\frac{\varepsilon kt}{1-2\delta}\leq C\varepsilon kt\le C\varepsilon,
\end{split}
\end{equation*}
due to the fact that $kt^{1+\delta}, \delta\in (0,\frac12)$ is increasing with respect to $t$ for any fixed $x$.
It is straightforward to check that $\phi_1$ satisfies
\begin{equation}\label{cont1}
\begin{split}
&\partial_t\phi_1=-\phi_1\left(\frac{k_t}{k}+\frac{2}{t}\right)+\varepsilon k,\\
&\partial_x\phi_1=-\frac{k_x}{k}\phi_1+2\varepsilon
\frac{\partial_xk(x, t_0)}{k(x, t_0)}\frac{k(x, t_0)t_0^2}{k(x,t)t^2}+
\frac{\varepsilon}{k(x, t)t^2}\int_{t_0}^t2k(x, s)\partial_xk(x, s)s^2ds,
\end{split}
\end{equation}
which, together with  (H2), yields that $|\partial_x\phi_1|\leq C\varepsilon$.
From $\eqref{r+s}_1$ and \eqref{cont1}, we have
\begin{equation*}
\partial_\beta(r+s-\phi_1)=(r+s-\phi_1)
\left[-\left(\frac{2B_t}{B}+\frac{k_t}{k}\right)-BB_tk^2rs\right]+R_1,
\end{equation*}
where $$R_1=-\partial_\beta\phi_1-\phi_1
\left[\left(\frac{2B_t}{B}+\frac{k_t}{k}\right)+BB_tk^2rs\right]-k\tilde{s}
-2\left(\frac{B_x}{B}+\frac{3k_x}{2k}\right)krs+\frac{k_x}{2k}k(r^2+s^2).$$
Again using \eqref{cont1} and Lemma \ref{lemmaB}, we obtain for small $\varepsilon$ and $\mu$ that
\begin{equation}\label{r1}
R_1\leq-\varepsilon k(1-C\varepsilon-C\mu)<0.
\end{equation}
It should be noted that the  {\it a priori}   assumption \eqref{tilde} for the derivative $\tilde{s}$ plays key role  in the analysis  of \eqref{r1}.
Then we get
\begin{equation*}
\partial_\beta(r+s-\phi_1)\le (r+s-\phi_1)
\left[-\left(\frac{2B_t}{B}+\frac{k_t}{k}\right)-BB_tk^2rs\right]
\end{equation*}
which  yields $r+s\leq\phi_1$ since $\phi_1(t_0)=2\varepsilon\geq|r+s|(x, t_0).$
Similarly, we can prove $r+s\geq-\phi_1$. Thus $|r+s|\leq\phi_1$ holds. Note that $\phi_1$ also satisfies
\begin{equation}
\phi_1\le 2\varepsilon
+\varepsilon\int_{t_0}^tk ds\le a_0\varepsilon.
\end{equation}
It then holds that
\begin{equation}\label{rjias}
|r+s|\leq \min\{a_0\varepsilon, ~~C\varepsilon kt\}.
\end{equation}
\par
For $r-s,$ let
$$\phi_2=\varepsilon\left(2+\int_{t_0}^tk(x, s)ds\right),$$
then $|\phi_2|\leq a_0\eps.$
From $\eqref{r-s}_1,$ we get
\begin{equation*}
\partial_\beta(r-s-\phi_2)=Q(r-s-\phi_2)+R_2,
\end{equation*}
where $$R_2=-\partial_\beta\phi_2+\phi_2Q+k\tilde{s}.$$
Since $\partial_t\phi_2=\varepsilon k$ and
$$|\partial_x\phi_2|=\varepsilon \left|\int_{t_0}^tk_xds\right|\leq C\varepsilon \int_{t_0}^tkds\leq C\varepsilon,$$
the  formula \eqref{Q} for $Q$ implies $|Q|\leq C\eps k$. It then holds that
\begin{equation*}
R_2\leq-\varepsilon k(1-C\mu-C\varepsilon)<0,
\end{equation*}
by choosing $\varepsilon$ and $\mu$ small enough. Therefore
\begin{equation*}
\partial_\beta(r-s-\phi_2)<Q(r-s-\phi_2),
\end{equation*}
which leads to $r-s\leq\phi_2$ due to  $\phi_2(t_0)=2\varepsilon\geq|r-s|(x, t_0).$  Similarly, $r-s\geq-\phi_2$ holds. Thus,
$$|r-s|\leq\phi_2\leq a_0\varepsilon.$$
Therefore, we have
\begin{equation*}
|r|\leq\frac{|r+s|}{2}+\frac{|r-s|}{2}\leq a_0\varepsilon,\quad |s|\leq\frac{|r+s|}{2}+\frac{|r-s|}{2}\leq a_0\varepsilon.
\end{equation*}

{\bf Step 2.}\quad
Before deriving the estimates on $\tilde{r}$ and $\tilde{s}$,
we introduce a useful comparison principle for hyperbolic system, which
is a variant   of Lemma 2.2 of \cite{H}. 
\begin{lemma}[Hong \cite{H}] \label{lemmasystem}
	Let $(u,v)$ be the $C^1$ smooth solutions to (\ref{ghs}) in $\Delta_P$ with the following conditions:
	\begin{equation*}
	\begin{split}
	&a_{12}(x,t)\le 0, \quad a_{21}(x,t) \leq0\text{ in } \Delta_P,\\
	&\varphi_1(x)\leq0\,  (\geq0),\quad \varphi_2(x)\geq0\,  (\leq0) \text{ for all } x\in [\Gamma_2(t_0;x^*, t^*),\Gamma_1(t_0; x^*, t^*)],\\
	& R_1(x,t)<0\,  (>0), \quad R_2(x,t)>0\,  (<0) \text{ in } \Delta_P.
	\end{split}
	\end{equation*}
	Then for $(x,t)\in\Delta_P$ with $t>t_0,$ $u(x,t)<0\, (>0),$ $v(x,t)>0\,  (<0).$
\end{lemma}

Now we define another control function
$$\phi_3=\mu\varepsilon\left(2+\int_{t_0}^tk(x,s)ds\right),$$
which satisfies $\phi_3\leq a_0\mu\varepsilon.$
Then we can derive the following equations for $\tilde{r}-\phi_3$ and $\tilde{s}+\phi_3$:
\begin{equation*}
\begin{split}
&\partial_\beta(\tilde{r}-\phi_3)=(Q+f_r-k_xs)(\tilde{r}-\phi_3)
+f_s(\tilde{s}+\phi_3)+R_3,\\
&\partial_\alpha(\tilde{s}+\phi_3)=g_r(\tilde{r}-\phi_3)
+(Q+g_s-k_xr)(\tilde{s}+\phi_3)+R_4,
\end{split}
\end{equation*}
where
\begin{equation*}
\begin{split}
&R_3=-\partial_\beta\phi_3+\phi_3(Q+f_r-k_xs-f_s)+(r-s)\delta_xf,\\
&R_4=\partial_\alpha\phi_3+\phi_3(g_r-Q-g_s+k_xr)+(r-s)\delta_xg.
\end{split}
\end{equation*}
From \eqref{deltaxf} and
the  estimate  $|r+s|\le C\varepsilon kt$, one has
$$|\delta_xf|\leq \frac{C}{t}|r+s|+C\varepsilon^2k\leq C\varepsilon k.$$
Similarly, we have $|\delta_xg|\leq C\varepsilon k.$
A direct  calculation yields that $\dt\phi_3=\mu\varepsilon k$ and
$$|\dx\phi_3|=\mu\varepsilon \int_{t_0}^t|k_x|ds\leq C\mu\varepsilon \int_{t_0}^tkds\leq C\mu\varepsilon.$$
From \eqref{fr},  $f_r-f_s$ is bounded by $C\eps k.$
Since $|Q|\leq C\eps k$  and $|k_xs|\leq C\eps k$, then
\begin{equation*}
\begin{split}
&R_3\leq(-\mu+C\varepsilon+C\mu\eps)\eps k<0,\\
&R_4\geq(\mu-C\varepsilon-C\mu\eps)\eps k>0,
\end{split}
\end{equation*}
if $\varepsilon$ is smaller than $\mu$.\par
Noting that $\dt \ln (kt^{1+\delta}) \ge 0$,   one has
$$f_s=-\left(\frac{B_t}{B}+\frac{k_t}{2k}\right)+h.o.t
\leq-\frac{1}{t}+o\left(\frac{1}{t}\right)+\frac{1+\delta}{2t}=-\frac{1-\delta}{2t}+o(\frac{1}{t})\leq0.$$
Similarly we can prove $g_r\leq0.$
The initial data satisfies the following estimates:
$$|\tilde{r}(x, t_0)|\leq2\mu\varepsilon=\phi_3(t_0) \text{ and }
|\tilde{s}(x, t_0)|\leq2\mu\varepsilon=\phi_3(t_0).$$
Then,  from Lemma \ref{lemmasystem}, we have $\tilde{r}\leq\phi_3$ and $\tilde{s}\geq-\phi_3.$ Similarly we can deduce that  $\tilde{r}\geq-\phi_3,$ and $\tilde{s}\leq\phi_3$ by considering the equations for $\tilde{r}+\phi_3$ and $\tilde{s}-\phi_3.$
Therefore we have
$$|\tilde{r}|\leq\phi_3\leq a_0\mu\varepsilon\text{ and }
|\tilde{s}|\leq\phi_3\leq a_0\mu\varepsilon.$$
\end{proof}
\begin{remark}
The estimate $$|r+s|\le C\varepsilon kt$$ implies that $r+s$ decays to zero  due to $k=o(\frac{1}{|t|})$, while $r-s$ may not decay. This means that $w+z$ decays faster than $w-z$ since $w=kr, z=rs$. Thus, we have verified this property observed in the ODE system \eqref{Riemann ODE} even for the general Gauss-Codazzi system \eqref{rs}. This property plays an essential role to close the  {\it a priori}   assumptions of derivatives $\tilde{r},\tilde{s}$ as shown above.
\end{remark}

Next we estimate the lower bound of $r-s$ and  upper bound of $r_x,  s_x$.
\begin{lemma}\label{lemmalow}
Under the conditions of Lemma \ref{lemmabound}, and if for any $x\in I(t_0)$,
$$(r-s)(x,t_0)\geq\varepsilon,$$
then in $\Delta_P$  for $t>t_0$ with  $t_0$ sufficiently large,
\begin{equation*}
\begin{split}
&(r-s)(x,t)\geq\frac{1}{2}\varepsilon,\\
&|\partial_xr(x, t)|\leq 2a_0\mu, \quad |\partial_xs(x, t)|\leq 2a_0\mu,\\
&|\partial_tr(x, t)|\leq C\varepsilon (k+t|k_t|), \quad |\partial_ts(x, t)|\leq C\varepsilon (k+t|k_t|).
\end{split}
\end{equation*}
\end{lemma}
\begin{proof}
Let
$$\phi_4=\varepsilon\left(1-3a_0\mu\int_{t_0}^tk(x, s)ds\right),$$
then for small $\mu$, $$\varepsilon\geq\phi_4\geq\frac{1}{2}\varepsilon.$$
From  $\eqref{r-s}$, we get
$$\partial_\beta(r-s-\phi_4)=Q(r-s-\phi_4)+R_5,$$
where  $$R_5=-\partial_\beta\phi_4+Q\phi_4+k\tilde{s}.$$
Again we know that  $\partial_t\phi_4=-3a_0\mu\varepsilon k$ and
$$|\partial_x\phi_4|=3a_0\mu\varepsilon\left|\int_{t_0}^tk_xds\right|\leq C\mu\varepsilon .$$
Thus, from Lemma \ref{lemmabound},
\begin{equation*}
\begin{split}
R_5&\geq3a_0\mu\varepsilon k-C\mu\varepsilon^2k-a_0\mu\varepsilon k-C\varepsilon^2k\\
&=\varepsilon k(2a_0\mu-C\varepsilon\mu-C\varepsilon)>0,
\end{split}
\end{equation*}
by choosing $\varepsilon$ sufficiently smaller than $\mu$. Therefore,
\[
\partial_\beta(r-s-\phi_4)>Q(r-s-\phi_4),
\]
which gives
 $$r-s\geq\phi_4\geq\frac{1}{2}\varepsilon.$$

Finally, we obtain
\begin{equation*}
\begin{split}
&|\partial_xr(x, t)|\leq\frac{|\tilde{r}|}{r-s}
\leq\frac{2a_0\mu\varepsilon}{\varepsilon}=2a_0\mu,\\
&|\partial_xs(x, t)|\leq\frac{|\tilde{s}|}{r-s}
\leq\frac{2a_0\mu\varepsilon}{\varepsilon}=2a_0\mu.
\end{split}
\end{equation*}
Note that
$$r_t=f-ksr_x\text{ and } s_t=g-krs_x,$$
then it is straightforward to deduce that
\begin{equation*}
\begin{split}
&|r_t|\leq \left|\frac{B_t}{B}+\frac{k_t}{2k}\right||r+s|+ C\varepsilon k\leq C\varepsilon(k+t|k_t|),\\
&|s_t|\leq \left|\frac{B_t}{B}+\frac{k_t}{2k}\right||r+s|+ C\varepsilon k\leq C\varepsilon(k+t|k_t|).
\end{split}
\end{equation*}
This completes the proof.
\end{proof}


We now prove the Theorem \ref{thm1}.

\begin{proof}[Proof of Theorem \ref{thm1}]
First, we choose sufficiently large $T=t_0$ such that Lemma \ref{lemmabound} and Lemma \ref{lemmalow} hold.
Consider the Cauchy problem for system (\ref{rs}) with initial data $ r(x,0)=\eta_0, s(x,0)=-\eta_0$.
Denote
$$H=1+\sup_{\substack{|t|<T\\ x\in\R}}\{|\dx^i\dt\ln B|, |\dx^i\dt\ln k|, |\dx^i(k\dx\ln B)|, |\dx^ik|,|\dx^2k|, |\dx^i(BB_tk^2)|, i=0,1\}.$$
Let $\eta_0$ be a small constant such that
$$\eta_0\exp(30HT)\leq\varepsilon<1.$$
From Lemma \ref{lemmageneral}, one has
$$|r(x,t)|\leq\varepsilon, \, |s(x,t)|\leq\varepsilon;\quad
 |\partial_x r(x,t)|\leq\varepsilon\le \sqrt{\varepsilon}=:\mu, \, |\partial_x s(x,t)|\leq\varepsilon\le \sqrt{\varepsilon}=:\mu$$
for all $x$ in $\R,$ $0<t\leq T$. On the other hand, integrating (\ref{r-s}) together  with the above estimates yields the following:
$$(r-s)(x,T)\geq\frac{\eta_0}{C(T)}\geq\frac{\varepsilon}{C(T)}$$
for some constant $C(T)$.
Then by  Lemmas \ref{lemmabound} and \ref{lemmalow}, there exists a unique solution to the system (\ref{rs}) with initial data $r(x,T), s(x,T)$ in the region
$t>T$ satisfying
\begin{equation*}
\begin{split}
&-a_0\varepsilon\leq s(x, t)<r(x, t)\leq a_0\varepsilon, \\
&|(r-s)\partial_xr(x,t)|\leq a_0\mu\varepsilon,  \quad |(r-s)\partial_xs(x, t)|\leq a_0\mu\varepsilon,\\
&|\partial_tr(x,t)|\leq C\varepsilon(k+t|k_t|),  \quad |\partial_ts(x, t)|\leq C\varepsilon(k+t|k_t|),\\
&(r-s)(x,t)\geq\frac{\varepsilon}{2C(T)},
\end{split}
\end{equation*}
for $(x,t)\in\R\times[T, \infty).$
Thus we can extend the solutions to the whole upper plane and get that
 $r(x,t), s(x,t)\in C^1(\R\times[0,\infty))$ with $r-s>0$ everywhere. For the
lower half plane, we can also get  $r(x,t), s(x,t)\in C^1(\R\times(-\infty, 0])$ in the same way. Therefore, the existence of the system (\ref{rs}) in the whole plane
 with initial data $r(x,0)=\eta_0,$ $s(x,0)=-\eta_0$ is proved.
The argument in Section \ref{ode} and the relation $w=kr,$  $z=ks$
yield the global existence of  $C^1$ solutions to the system  (\ref{GC}).
Furthermore, the fundamental theorem of surface theory  implies that there exists a $C^3$ surface with the prescribed metric whose Gauss curvature satisfying (H1)-(H3). This completes the proof of Theorem \ref{thm1}.
\end{proof}

\bigskip

\section{Proof of Theorem \ref{thm} in Geodesic Polar Coordinates}\label{geodesicpolar}

In this section, we shall apply the  Theorem \ref{thm1}  on the existence of isometric immersion in geodesic coordinates $(x,t)$  to prove the main Theorem \ref{thm} in geodesic polar coordinates $(\theta,\rho)$.   The idea of the proof follows  Hong \cite{H} (cf. \cite{HH}) closely with  the modifications that are necessary due to the slower decay rate at  infinity of  the Gauss curvature.  For the sake of completeness we shall provide an outline of the proof.

The metric under the geodesic polar coordinates $(\theta, \rho)$ is of the form
$$\mathfrak{g}=G^2(\theta, \rho)d\theta^2 +d\rho^2$$  satisfying
\begin{eqnarray}
\left\{
\begin{array}{lllll}
G_{\rho\rho}=k^2G,\\
G(\theta, 0)=0, \, G_{\rho}(\theta, 0)=1, \text{ for any } \theta\in[0, 2\pi].
\end{array}
\right.
\end{eqnarray}
Then the system \eqref{wz} in the polar coordinates $(\theta,\rho)$ becomes
 \begin{equation}\label{Riemannpolar}
 \begin{split}
&\bar{w}_\rho+\bar{z}\bar{w}_\theta=-\left(\frac{G_\rho}{G}-\frac{k_\rho}{2k}\right)\bar{w}-
\left(\frac{G_\rho}{G}+\frac{k_\rho}{2k}\right)\bar{z}-GG_\rho \bar{w}^2\bar{z}
-\left(\frac{G_\theta}{G}+\frac{k_\theta}{2k}\right)\bar{w}\bar{z}+\frac{k_\theta}{2k}\bar{w}^2,\\
&\bar{z}_\rho+\bar{w}\bar{z}_\theta=-\left(\frac{G_\rho}{G}-\frac{k_\rho}{2k}\right)\bar{z}-
\left(\frac{G_\rho}{G}+\frac{k_\rho}{2k}\right)\bar{w}-GG_\rho \bar{w}\bar{z}^2
-\left(\frac{G_\theta}{G}+\frac{k_\theta}{2k}\right)\bar{w}\bar{z}+\frac{k_\theta}{2k}\bar{z}^2,
\end{split}
\end{equation}
which is equivalent to the system (\ref{GC}) for smooth
solutions if $\bar{w}>\bar{z}$, where we use $\bar{w}(\theta,\rho), \bar{z}(\theta,\rho)$ to denote $w(x,t), z(x,t)$ in the geodesic polar coordinates $(\theta,\rho)$.

First we have the following properties of $G$  in the polar coordinates similar to Lemma \ref{lemmaB} and the proof is omitted.
\begin{lemma}\label{lemmaG}
Assume that the assumptions (A1)-(A2) hold, then
\begin{equation*}
\begin{split}
&\frac{G_\rho}{G}=\frac{1}{\rho}+o\left(\frac{1}{\rho}\right),\\
&\partial_\theta^i\ln G, i=1, 2,  ~\rho\partial_\rho\partial_\theta\ln G\text{ are bounded for sufficiently large } \rho,\\
&1\leq\partial_\rho G\leq b_0,~\rho\leq G\leq b_0\rho,~\rho\partial_\rho\ln G>1,
\end{split}
\end{equation*}
where $b_0=\exp\left\{\sup_{\theta}\int_0^\infty\rho k^2d\rho\right\}.$
\end{lemma}

As in Section \ref{geodesic}, we also make the following variable transformations:
$$\bar{r}=\frac{\bar{w}}{k},\quad \bar{s}=\frac{\bar{z}}{k},$$
 then  (\ref{Riemannpolar}) can be rewritten as
\begin{equation}\label{rsbar}
\begin{split}
\bar{r}_\rho+k\bar{s}\bar{r}_\theta&=-\left(\frac{G_\rho}{G}+\frac{k_\rho}{2k}\right)(\bar{r}+\bar{s})
      -\left(\frac{G_\theta}{G}+\frac{3k_\theta}{2k}\right)k\bar{r}\bar{s}+
                        \frac{k_\theta}{2k}k\bar{r}^2-GG_\rho k^2\bar{r}^2\bar{s}\\
                        &=:f(\bar{r}, \bar{s}, \theta, \rho),\\
\bar{s}_\rho+k\bar{r}\bar{s}_\theta&=-\left(\frac{G_\rho}{G}+\frac{k_\rho}{2k}\right)(\bar{r}+\bar{s})
       -\left(\frac{G_\theta}{G}+\frac{3k_\theta}{2k}\right)k\bar{r}\bar{s}+
                        \frac{k_\theta}{2k}k\bar{s}^2-GG_\rho k^2\bar{s}^2\bar{r}\\
                        &=:g(\bar{r}, \bar{s}, \theta, \rho).
\end{split}
\end{equation}
The system \eqref{rsbar} has the same form as the system \eqref{rs} for which the Theorem \ref{thm1} holds. However,    Theorem \ref{thm1} can not be directly   applied to the system \eqref{rsbar} due to the following two reasons: (i) The condition (H3) fails to be true,  and (ii) The coefficients in the  system \eqref{rsbar} are singular at the center $\rho=0$.

To prove the main Theorem \ref{thm}, we first consider the Gauss-Codazzi system in  a neighbourhood of the center $\rho=0$, denoted by $\Omega_1$
in the geodesic coordinates, and then solve the Gauss-Codazzi system in $\Omega_2$ that is the complement of $\Omega_1$ in $\mathbb{R}^2$ in  the geodesic polar coordinates.

 As in previous section, we only consider the part $t\geq0.$ We construct the domains $\Omega_1$ and $\Omega_2$ as follows. Set
\begin{equation*}
\begin{split}
t_0(x)=R(1+x^2), \quad \Omega_1=\{(x, t)| 0\leq t\leq t_0(x)\}, \quad \Omega_2=\R_+^2\setminus\Omega_1,
\end{split}
\end{equation*}
where $R$ is a constant to be determined.  Then we have
$$\partial\Omega_2=\{(x, t)| t=t_0(x)\}.$$
Note that $\Omega_1$ constructed above is different from that of \cite{HH} or \cite{H}.

Similar to Lemma \ref{lemmabound} and Lemma \ref{lemmalow}, we can prove the existence of global smooth  solution to any generalized Cauchy problem in the region $\{\rho>R\}$ for (\ref{rsbar}). Precisely, suppose that $\Omega$ is an unbounded domain with a smooth boundary curve:
\begin{equation*}
\Omega=\{(\theta, \rho); \theta_1(\rho)<\theta<\theta_2(\rho), \rho>R\}\subset
\{(\theta, \rho): \rho>R, 0<\theta<\pi\},
\end{equation*}
for some smooth functions $\theta_1<\theta_2$ in $[R, +\infty).$ We have the following lemma.
\begin{lemma}\label{lemmaexist}
Assume that  $\partial\Omega$ is space-like with respect to $\rho$,  i.e.,
\begin{equation*}
\partial_\rho\theta_1(\rho)<k(\theta_1(\rho), \rho)\bar{s}(\theta_1(\rho), \rho),
~\partial_\rho\theta_2(\rho)>k(\theta_2(\rho), \rho)\bar{r}(\theta_2(\rho), \rho),
\end{equation*}
and on $\partial\Omega,$
\begin{equation}\label{initialrsbar}
\begin{split}
-\varepsilon\leq\bar{s}<\bar{r}\leq \varepsilon,\quad
\bar{r}-\bar{s}\geq\varepsilon,\quad
|\partial_\theta\bar{r}|, |\partial_\theta\bar{s}|\leq\mu.
\end{split}
\end{equation}
Then there are two constants $R_0$ and $\eps_0$ such that
when $R\geq R_0$ and $0<\varepsilon<\varepsilon_0,$ there exists a global smooth solution $(\bar{r}, \bar{s})$ to the problem (\ref{rsbar})-(\ref{initialrsbar}) in the closure of $\Omega$.
\end{lemma}
Lemma \ref{lemmaexist} can be applied for $\Omega_2,$ provided that the conditions of Lemma \ref{lemmaexist} are satisfied. On the other hand, as in Lemma \ref{lemmageneral}, we can find a small smooth solution  in  $\Omega_1$ as long as the initial data are sufficiently small. To match the solutions in $\Omega_1$ and  $\Omega_2$, we use a coordinate transformation $F$ from the geodesic coordinates $(x, t)$ to the geodesic polar coordinates $(\theta, \rho).$  As in  \cite{HH} or \cite{H}, we have the following conclusions,  the proofs of which are omitted:
\begin{itemize}
\item[(1)] $t, \frac{1}{2}|x|\leq \rho(x, t)\leq t+|x|,$ and then $t+|x|\leq3\rho.$
\item[(2)] $\rho_t=\tanh \Phi, ~\theta_t=\frac{\xi}{G\cosh\Phi}, ~\rho_x=-\frac{\xi B}{\cosh\Phi}, ~\theta_x=\frac{B}{G}\tanh\Phi$  with $\Phi=\int_0^t\partial_\rho\ln Gds$ for $x\ne0.$
\item[(3)] The change of functions $(r, s)\mapsto(\bar{r}, \bar{s})$ satisfies
\begin{equation*}
\begin{split}
(\rho_t+kr\rho_x)k\bar{r}=\theta_t+kr\theta_x,\quad
(\rho_t+ks\rho_x)k\bar{s}=\theta_t+ks\theta_x,
\end{split}
\end{equation*}
and it makes sense in a neighbourhood of $\partial \Omega_2$ in $\Omega_1$, provided that $R$ is large enough.
\item[(4)] For any normalized vector $V$, define the differential map
\begin{equation*}
F_*(V)=(\rho_t+\zeta\rho_x)(\partial_\rho+\bar{\zeta}\partial_\theta)
\end{equation*}
with
\begin{equation*}
\bar{\zeta}=\frac{\theta_t+\zeta\theta_x}{\rho_t+\zeta\rho_x}
\text{ if~} \rho_t+\zeta\rho_x\neq0.
\end{equation*}
And denote $\bar{F}_*(\zeta)=\bar{\zeta}$ if it makes sense. Then
it holds that
\begin{equation}\label{krbar}
k\bar{r}=\bar{F}_*(kr), ~~k\bar{s}=\bar{F}_*(ks),
\end{equation}
and
\begin{equation}\label{r-sbar}
\bar{r}-\bar{s}=\frac{B(r-s)}{G(\rho_t+kr\rho_x)(\rho_t+ks\rho_x)}.
\end{equation}
\end{itemize}
\
\
\\

To seek the initial data, we choose two smooth even functions $h_1$ and $h_2$ in $\R$ satisfying
 \begin{equation*}
 h_1(\zeta)\geq1+\sup_{\substack{(x, t)\in\Omega_1\\  |x|\leq\zeta\\ i=0, 1}}
 \{|\dx^i\dt\ln B|, |\dx^i\dt\ln k|, |\dx^i(k\dx\ln B)|,|\dx^2 k|, |\dx^ik|, |\dx^i(BB_tk^2)|\}, 
 \end{equation*}
 and
 \begin{equation*}
 h_2(\zeta)\geq1+h_1(\zeta)+\sup_{(x, t)\in\Omega_1, |x|\leq\zeta}
 \left\{\frac{|t_0'(x)|}{t_0(x)+R}\right\},
 \end{equation*}
and increasing in $\R_+.$
Then we define $\eta(x)$ for $x>0$ by
\begin{equation*}
\eta(x)=\frac{1}{8\pi}\int_x^\infty\psi(\zeta)
\frac{\exp\{-30(t_0(\zeta+1)+R)h_1(\zeta+1)-\zeta^2\}}
{(t_0(\zeta+1)+R)h_2(\zeta+1)}d\zeta,
\end{equation*}
where $\psi$ is a smooth cutoff function such that $\psi(\zeta)=1$ for $\zeta\geq1$ and $\psi(\zeta)=0$ for $\zeta$ near zero.  And we can  extend $\eta$ as a smooth even function in $\R.$  Finally, the initial data is chosen as
\begin{equation}\label{initialrs}
r(x, 0)=\sigma\eta(x), ~~s(x, 0)=-\sigma\eta(x),
\end{equation}
with small positive constant $\sigma$ to be determined later. Thus, for the Cauchy problem for  (\ref{rs}) with \eqref{initialrs} in $\Omega_1$, we obtain the following lemma.
\begin{lemma}\label{lemmainner}
There exists a positive $\sigma_0\in(0, 1)$ such that the Cauchy problem of (\ref{rs}) with (\ref{initialrs}) in $\Omega_1$
admits a smooth solution $(r, s)$ with $r>s$ for any $\sigma\in(0, \sigma_0).$
Furthermore, on $\partial \Omega_2,$
the solution satisfies (\ref{initialrsbar}).
\end{lemma}

Once we have shown Lemma \ref{lemmainner}, we can complete the proof of Theorem \ref{thm} as in \cite{HH} or \cite{H}.
In fact, assume  $(r, s)$ is the solution of (\ref{rs}) with (\ref{initialrs}) as in Lemma \ref{lemmainner}. Since the change of functions $(r, s)\mapsto(\bar{r}, \bar{s})$ makes sense in a neighbourhood of $\partial \Omega_2$ in $\Omega_1$, and Lemma \ref{lemmainner} guarantees that the conditions in Lemma \ref{lemmaexist} are satisfied for $\Omega=F(\Omega_2),$  
there is a smooth solution $(\bar{r}, \bar{s})$ to the problem (\ref{rsbar})-(\ref{initialrsbar}).
 Taking the variable transformation:
 $$(\theta, \rho)\in F(\Omega_2)\mapsto(x, t)\in\Omega_2 \text{ and } (\bar{r}, \bar{s})\mapsto(r, s),$$
again from Lemma \ref{lemmaexist},   we have a smooth extension of $(r, s)$ in $\Omega_1$ to $\Omega_2,$ with $r>s$ everywhere.  Therefore, we obtain a global smooth solution of (\ref{rs}) with (\ref{initialrs}) with $r>s$ for $t>0,$ 
which yields a smooth isometric immersion of $g$ into $\R^3$ by applying fundamental  theorem of surface theory. Therefore it  remains to verify Lemma \ref{lemmainner},
the proof of which is different from \cite{HH} or \cite{H}.
\\

{\textit{\textbf{The proof of Lemma \ref{lemmainner}: }}
Similarly to \cite{HH}, it is easy to verify $r>s$ on $\partial\Omega_2$ and 
$\bar{r}>\bar{s}$ on $\partial F(\Omega_2).$
We can show that the differential system (\ref{rs}) with (\ref{initialrs}) in $\Omega_1$ has a global
smooth solution $(r, s)$ for each $x$ and each $\sigma\in(0, \sigma_0)$ with $\sigma_0\leq1,$
satisfying
\begin{equation}\label{rsomega1}
|r|, |s|, |r_x|, |s_x|\leq\frac{C\sigma}{(t_0(x)+R)h_2(x)}
\end{equation}
in the characteristic triangle where the solution exists. Furthermore, we can also prove that $\partial\Omega_2$ is space-like in $\rho$. The proof is omitted.

To complete the proof of Lemma \ref{lemmainner}, it remains to show that $(r, s)$ satisfies (\ref{initialrsbar}).
From (\ref{krbar}), we have
\begin{equation*}
\begin{split}
\bar{r}=\frac{\bar{F}_*(kr)}{k}=\frac{\xi+kBr\sinh\Phi}{kG(\sinh\Phi-kBr\xi)},\quad
\bar{s}=\frac{\bar{F}_*(ks)}{k}=\frac{\xi+kBs\sinh\Phi}{kG(\sinh\Phi-kBs\xi)}.
\end{split}
\end{equation*}
Since
$0\leq1-\rho_t\leq\frac{1}{\cosh\Phi},$
then, for $|x|>2,$
\begin{equation*}
\begin{split}
\Phi=\int_0^t\frac{G_\rho}{G}ds\geq\int_0^t\frac{1}{\rho}ds
\geq\int_0^t\frac{\rho_s}{\rho}ds=\ln\rho(x, t)-\ln\rho(x, 0),
\end{split}
\end{equation*}
and thus $e^{\Phi}\geq\frac{\rho(x, t)}{\rho(x, 0)}.$
From (A1), we also have
$k\geq\frac{C}{\rho^{1+\delta}},$  and then  on $\partial\Omega_2\cap\{|x|>2\},$
\begin{equation*}
\begin{split}
|\bar{r}|\leq\frac{1}{2kG\sinh\Phi}\leq\frac{C\rho(x, 0)
\rho^{1+\delta}(x, t)}{\rho^{2}(x, t)}
\leq\frac{C|x|}{t_0(x)^{1-\delta}}\leq\frac{C|x|}{R^{1-\delta}|x|^{2-2\delta}}
\leq\frac{C}{R^{1-\delta}}\leq\eps,
\end{split}
\end{equation*}
due to $0<\delta<1/2.$
Meanwhile, on  $\partial\Omega_2\cap\{|x|\leq2\},$ one has
\[
|\bar{r}|\leq\frac{1}{2kG\sinh\Phi}\leq\frac{C\rho^\delta}{R}
\leq\frac{C+CR^\delta}{R}\leq\eps.
\]
Therefore, on $\partial\Omega_2,$ $|\bar{r}|\leq\eps.$ Similarly,
$|\bar{s}|\leq\eps.$

We also need to estimate $\partial_\theta\bar{r}$ and $\partial_\theta\bar{s}.$
Taking the partial derivative of $\bar{r}$
with respect to $\theta,$ we get on $\partial\Omega_2,$
\begin{equation*}
\begin{split}
\partial_\theta\bar{r}&=-\left(\frac{k_\theta}{k}+\frac{G_\theta}{G}\right)
\frac{1}{kG}\frac{\xi+kBr\sinh\Phi}{\sinh\Phi-kBr\xi}\\
&\quad +\partial_\theta r\left[\frac{B\sinh\Phi}{G(\sinh\Phi-kBr\xi)}
+\frac{B\xi(\xi+kBr\sinh\Phi)}{G(\sinh\Phi-kBr\xi)^2}\right]\\
&\quad + \left(\frac{B_x}{B}x_\theta+\frac{B_t}{B}t_\theta+\frac{k_\theta}{k}\right)
\left[\frac{rB\sinh\Phi}{G(\sinh\Phi-kBr\xi)}-\frac{B\xi(\xi+kBr\sinh\Phi)}{G(\sinh\Phi-kBr\xi)^2}\right]\\
&\quad +\partial_\theta\Phi\left[\frac{Br\cosh\Phi}{G(\sinh\Phi-kBr\xi)}
-\frac{\cosh\Phi(\xi+kBr\sinh\Phi)}{kG(\sinh\Phi-kBr\xi)^2}\right]=\sum_{j=1}^4I_j.
\end{split}
\end{equation*}
For $I_1,$ as above one has
$|I_1|\leq \frac{C}{kG\sinh\Phi}\leq\frac{C}{R^{1-\delta}}\leq\mu,$
with $R$ large enough.
For $I_2,$  we have  on $\partial F(\Omega_2)$,
\begin{equation*}
\begin{split}
|\partial_\theta r|&=|r_tt_\theta+r_xx_\theta|\leq|f(\theta, \rho)t_\theta|+k|s||r_x||t_\theta|+|r_x||x_\theta|\\
&\leq\frac{C\sigma}{(t_0(x)+R)h_2(x)}
\left[\frac{G}{B}\tanh\Phi+\frac{(h_1(x)+k)G}{\cosh\Phi}\right]\leq C\sigma,
\end{split}
\end{equation*}
since $G\leq b_0\rho,$ $\rho\leq 2t$ on $\partial\Omega_2.$ Thus we obtain
$|I_2|\leq C|\partial_\theta r|\leq C\sigma\leq\mu$
from \eqref{rsomega1} and choosing small $\sigma_0$. In the same way,
we have $|I_3|\leq C|r|\leq\mu$.
For $I_4,$ it holds that
$$|I_4|=\left|\partial_\theta\Phi
\frac{-\xi\cosh\Phi(1+k^2B^2r^2)}{kG(\sinh\Phi-kBr\xi)^2}\right|
\leq \frac{C\cosh\Phi}{kG\sin^2\Phi}|\partial_\theta\Phi|.$$
We then derive  the estimates of $\partial_\theta\Phi$ as follows. For $|x|>2,$
\begin{equation*}
\begin{split}
\partial_\theta\Phi&=\int_0^t(\partial_\rho\ln G)_xds x_\theta+\partial_\rho\ln Gt_\theta\\
&=\left[\int_0^t(k^2-(\partial_\rho\ln G)^2)\frac{-\xi B}{\cosh\Phi}ds
+\int_0^t\partial_{\theta\rho}\ln G\frac{B\tanh\Phi}{G}ds\right]\frac{G}{B}\tanh\Phi
\\
&\qquad+\partial_\rho\ln G\frac{\xi G}{\cosh\Phi}.
\end{split}
\end{equation*}
Moreover, it holds that
\begin{equation*}
\begin{split}
&\left|\int_0^tk^2\frac{-\xi B}{\cosh\Phi}ds\right|\leq C\int_0^tk^2sds\leq C,\\
&\left|\int_0^t\partial_{\theta\rho}\ln G\frac{B\tanh\Phi}{G}ds\right|
\leq\int_0^t\frac{Ct}{\rho^3}\int_0^\rho k^2\alpha^2d\alpha ds
\leq\int_0^tCk^2sds\leq C,
\end{split}
\end{equation*}
and
\begin{equation*}
\left|\int_0^t\frac{G_\rho^2}{G^2}\frac{-\xi B}{\cosh\Phi}ds\right|
\leq\int_0^t\frac{C\rho(x, 0)s}{\rho(x, s)^3}ds
\leq C|x|\int_0^t\frac{1}{(s+|x|)^2}ds\leq C,
\end{equation*}
where we have used  the fact:
$t+|x|\leq 3\rho.$
Hence, for $|x|>2,$
\begin{equation*}
|\partial_\theta\Phi|\leq \frac{CG}{B}\tanh\Phi+\frac{C}{\cosh\Phi}.
\end{equation*}
Therefore, on  $\partial\Omega_2\cap\{|x|>2\},$
\begin{equation*}
\begin{split}
|I_4|&\leq\frac{C\cosh\Phi}{kG\sinh^2\Phi}
\left(\frac{CG}{B}\tanh\Phi+\frac{C}{\cosh\Phi}\right)\\
&\leq\frac{C}{kB\sinh\Phi}+\frac{C}{kG\sinh^2\Phi}
\leq\frac{Ct^\delta\rho(x, 0)}{\rho(x, t)}+
\frac{C\rho^2(x, 0)}{\rho^{2-\delta}}
\leq\frac{C}{R^{1-\delta}}\leq\mu,
\end{split}
\end{equation*}
for large $R$ and small $\sigma_0$.

When $|x|\leq2,$ on $\partial\Omega_2,$
\[
B_t(x, \infty)\geq\inf_{|x|\leq2}\int_0^tk^2(x, s)ds>0.
\]
Then following the idea in the proof of Lemma \ref{lemmaB}, we have, on  $\partial\Omega_2\cap\{|x|\leq2\},$
$$\left|\frac{B_t}{B}-\frac{1}{t}\right|\leq
\frac{C}{t}\frac{t^{-1}\int_0^tBk^2sds}{B_t(x, \infty)}\leq Ck^2t,$$
and
\[
\frac{G_\rho}{G}-\frac{1}{\rho}
\leq\frac{C}{\rho}\frac{\rho^{-1}\int_0^\rho Gk^2\rho ds}{G_\rho(x, \infty)}
\leq Ck^2\rho.
\]
Recalling the formula of $\partial_\theta\Phi$ in \cite{H},
\begin{equation*}
\begin{split}
\partial_\theta\Phi
&=\xi\frac{G_\rho}{\cosh\Phi}+\bigg[\frac{1}{\rho}-\frac{1}{t}
+\left(\frac{G_\rho}{G}-\frac{1}{\rho}\right)\\
&~~~~+(\tanh\Phi-1)\frac{G_\rho}{G}
-\left(\frac{B_t}{B}-\frac{1}{t}\right)\bigg]\xi G\sinh\Phi,
\end{split}
\end{equation*}
we have on $\partial\Omega_2\cap\{|x|\leq2\},$
$$|\partial_\theta\Phi|\leq \frac{C}{R}+Ck^2R^2\sinh\Phi,$$
since $\rho\leq t_0(x)+|x|\leq CR$ and $\frac1{t}-\frac1{\rho}\leq C/R^2.$
Therefore, on $\partial\Omega_2\cap\{|x|\leq2\},$
\begin{equation*}
\begin{split}
|I_4|\leq\frac{C\cosh\Phi}{kG\sinh^2\Phi}
\left(\frac{C}{R}+Ck^2R^2\sinh\Phi\right)
\leq\frac{C}{kR^3}+CkR
\leq\mu,
\end{split}
\end{equation*}
where we have used the fact:
$\frac{C}{R^{1+\delta}}\leq k=o\left(\frac{1}{R}\right)$ on $\partial\Omega_2\cap\{|x|\leq2\},$
for $R$   large enough. Finally, we get $|I_4|\leq\mu$ on $\partial\Omega_2.$
Therefore,  $|\partial_\theta\bar{r}|\leq\mu\text{ on }\partial\Omega_2.$
Similarly, we have
$
|\partial_\theta\bar{s}|\leq\mu,
$ on $\partial\Omega_2.$
The proof is complete.


\cqfd

\bigskip
\section*{Acknowledgments}
F. Huang's research was supported in part  by NSFC Grant No. 11371349.
D. Wang's research was supported in part by the NSF Grant DMS-1312800 and NSFC Grant No. 11328102. The authors would like to thank Professor Qin Han for his valuable discussions and suggestions.
%


\end{document}